\documentclass[12pt]{article}
\newcommand{\nocopyright}{
No Copyright\thanks{
The authors hereby waive all copyright
and related or neighboring rights to this work,
and dedicate it to the public domain.
This applies worldwide.
}}
\title{Spectral invariants and playing hide-and-seek on surfaces}
\author{Peter G.\ Doyle \and Jean Steiner}
\date{Original dated 27 September 2005\\ minor revisions dated 2009
\\ \nocopyright
}

\usepackage{amsmath,amscd,amssymb, amsthm, color, graphics,  color,epsfig,bm,epsf,times}
\newcommand{\note}[1]{}
\theoremstyle{plain}
\newtheorem{theorem}{Theorem}[section]
\newtheorem{lemma}[theorem]{Lemma}
\newtheorem{prop}[theorem]{Proposition}
\newtheorem{cor}[theorem]{Corollary}
\theoremstyle{definition}
\newtheorem{definition}[theorem]{Definition}

\theoremstyle{remark}

\newtheorem{rmk}[theorem]{Remark}

\def\bi{\begin{itemize}}
\def\ei{\end{itemize}}
\def\ben{\begin{enumerate}}
\def\een{\end{enumerate}}
\def\be{\begin{equation*}}
\def\ee{\end{equation*}}
\def\bea{\begin{eqnarray*}}
\def\eea{\end{eqnarray*}}
\def\bml{\begin{multline*}}
\def\eml{\end{multline*}}
\def\ie{{\it i.e.\ }}

\newcommand{\dif}{\mathrm{d}}
\newcommand{\w}{w}

\newcommand{\1}{\mathbf{1}}
\newcommand{\ones}{\mathbf{\hat{1}}}
\newcommand{\rn}{\mathbb{R}^n}

\begin{document}
\maketitle
\begin{abstract}
We prove  the expected duration of a game of hide-and-seek played on a Riemannian manifold under the laws of Brownian Motion is a spectral invariant:  it is a zeta-regularized version of the `trace' of the Laplacian.   An analogous hide-and-seek game may be played on Markov chains, where the spectral invariant that emerges is a classical quantity known as Kemeny's constant.  We develop the analogies between the two settings in order to highlight the connections between the regularized trace and Kemeny's constant.   Our proof relies on  the connections between Green's functions and expected hitting times, and the fact that the regularized trace may be approached via the Green's function. 

\vspace{0.5cm}
\end{abstract}

\vspace{1em}
\begin{center}
\fbox{
\begin{minipage}{20em}
\noindent
\emph{
Why bother about journals? The arXiv is where math research lives.
}

\vspace{0.5em}
\noindent
\emph{
 ---  Oded Schramm, 12 April 2008
}
\end{minipage}
}
\end{center}
\vspace{1em}

\section{Introduction}
The (regularized) trace of the Laplacian is a geometric spectral invariant which has been shown to have interesting extremal properties.  In \cite{Morpurgo}, Morpurgo showed that among all metrics conformally equivalent to the standard round metric on the sphere, the regularized trace is minimized at the standard round metric.  More recently, in \cite{Okikiolutrace}, Okikiolu has shown that the infimum of the trace within the conformal class of any metric with a fixed volume is bounded above by the trace of the standard round sphere with the same fixed volume.   The trace is closely related to a sharp Sobolev inequality, as is laid out in both \cite{Morpurgo} and \cite{Okikiolutrace}, and in some sense these results suggest that the trace `hears' the sphericalness of the surface.  Analogous statements hold on higher dimensional manifolds for the trace of  the Paneitz operator, a  higher order elliptic operator, and the connection to a Sobolev inequality persists in the higher dimensional setting.

In this paper we will give a probabilistic interpretation of the trace.  We will describe two hide-and-seek games in two different settings.  The first setting is that of a  finite state Markov chain,  the operator of giving rise to our spectral invariant will be $I-P$, where $P$ is the transition matrix for the Markov chain, and the second setting is that of a surface, in which case the analogous operator is the Laplacian, $\Delta$.  In both settings, the operators $(I-P)$ and $\Delta$ have discrete eigenvalues, and the operators capture the difference between the value of the function at a point and a weighted  average at the   neighbors, and in both cases we will need to consider the appropriate Green's functions and approximate inverses to $(I-P)$ and $\Delta$, in order to analyze the hide-and-seek games.  In both settings, we will see that the expected duration of the first of our hide-and-seek games is a spectral invariant, and the expected duration of the second game is the density for the spectral invariant.    

Now, for the remainder of the introduction, we will describe the two hide-and-seek games, and then  interpret a spectral invariant on surfaces and Markov chains in terms of the expected durations of the games .     In each of the hide-and-seek  games we consider, one player's location is chosen deterministically and the  other  is chosen at random, according to the  equilibrium  measure.  In order to play this game on a finite state, discrete time Markov chain with $n$ states, given by $\mathcal{S}=\{1,2,\ldots, n \}$, the seeker will begin at a state $i\in \mathcal{S}$, and the hider will be at a state $j\in\mathcal{S}$.   On a surface, $M$, the seeker will begin at a point $x\in M$, and the hider will be at a point $y\in M$    On a surface, the probability of hitting a particular point is $0$, so the seeker can't find $y$.  However, the seeker can, with probability one,  hit a ball of radius $\epsilon$ around a point, so the hide-and-seek   games will end when the seeker, having started out  at the point $x$,  comes within  a distance of $\epsilon$ from the hider at $y$ (\ie  the seeker must reach the epsilon ball, $B_\epsilon(y)$).    The two games that will be considered are as follows:
 \begin{description}
 \item[Game I]\label{i.gameI}  The seeker is deterministically located at point $x\in M$, or at state $i\in \mathcal{S}$. The hider is  at a random location, $y\in M$ or $j\in \mathcal{S}$.
 \item [Game II]\label{i.gameII}  The seeker is randomly located at $x\in M$, or at state $i\in \mathcal{S}$ .  The hider is at deterministic   location, $y\in M$ or at state $j\in \mathcal{S}$.
 \end{description}
{\it A priori},  the expected duration of both games should depend on the deterministically chosen data (\ie the starting point for Game I, and the ending point for Game II), but we will see that Game I actually turns out to be independent of the deterministic data.

In order to give a more precise statement of our main results (and to give a bit of context for the interest in a spectral invariant such as the regularized trace), let's think spectrally about each of the two settings.
  In the case of surfaces, the Laplacian, $\Delta$ has a discrete spectral resolution with eigenvalues $\{\lambda_j\}_{j=0}^\infty$, and the eigenvalues are unbounded, with the lowest eigenvalue of $0$ corresponding to the constant functions, and asymptotic growth $\lambda_j\sim j$ for large $j$.   Through a regularization process, one can make sense of both the determinant and the trace of $\Delta$.    The determinant is a well-known spectral invariant that picks out geometric features such as constant curvature, since it is extremized at constant curvature as stated by  Osgood, Phillips and Sarnak, in \cite{Ops}.    While the story of the log-determinant  on surfaces, is reasonably well understood, investigations of the regularized trace have been undertaken more recently.  The  behavior of the regularized trace on the sphere is well-understood, but there is much to be understood about its behavior on other surfaces.     The regularized trace  is defined below in \eqref{e.defregzeta},  but heuristically one may think of it as $Tr'(\Delta^{-1})=\sum_{j\geq 1}' \frac{1}{\lambda_j}$, where we point out that $\Delta^{-1}$ only makes sense once one specifies its behavior on the null-space, and the prime reminds us that the sum does not actually converge.     In \cite{Morpurgo}, Morpurgo considers the  regularized trace and its behavior under conformal change, and he finds that it is extremized on the round sphere.  In \cite{steiner_duke}, Steiner gives a characterization of the density for the regularized trace.  In work of Chiu, in \cite{Chiu}, the regularized trace is shown to be linearly related   to the log-determinant for constant curvature tori, though the `trace' is not identified as such.  Additional work towards understanding the behavior of the regularized trace appears in work of Steiner and Doyle  in    \cite{skinnytori}, and also in forthcoming work by Okikiolu in \cite{Okikiolutrace}.     Additionally the density for the regularized trace  in connection with a condition for criticality of the log-determinant in work of Okikiolu \cite{Okikiolu} and Richardson \cite{Richardson}.

In this paper, our discussion of spectral invariants on surfaces will be complemented by a parallel discussion which takes place on a finite state, discrete time Markov chain.    If the Markov chain is irreducible and ergodic, the transition matrix $P=(p_{ij})$  will have eigenvalues $1> \nu_1\geq \nu_2 \geq \cdots \nu_{n-1}$.   Spectral invariants have also been considered  on Markov chains, see for example the survey of Aldous and Fill in \cite{Aldousfill}, and also work of Diaconis and Stroock in \cite{Diaconisstroock},  Levene and Loizou in \cite{Leveneloizou}, \footnote{There are many important references absent from this list.}.    The spectral invariant in the present discussion is the  `trace' for the operator $(I-P)^{-1}$, which  will need to be suitably regularized, since the operator $(I-P)$ has a one-dimensional null-space.  It turns out that a quantity called Kemeny's constant, which is denoted by $K^i=K$ and defined below in \eqref{e.kemconst} is exactly this regularized trace.    The quantity $K^i$, given below in \eqref{e.kemconst} is an expression given in terms of expected first hitting times, and, at first glance,  it appears  that $K^{i}$ depends on $i$.  However, $K^i$  was observed by Kemeny to be constant in \cite{Kemenysnell}.       It turns out that Kemeny's constant is essentially the regularized trace of the inverse, \ie $Tr(I-P)^{-1}=\sum_{j=1}^{n-1} \frac{1}{1-\nu_j}$ and this perspective is put forth by Meyer in \cite{Meyer}.  Since this quantity appears in a number of places, we point out that in his notes, \cite{Aldousfill}  Aldous discusses the constancy of $K^i$, which  is referred to in the `Random target identity', a phrase coined by Lov\'asz and Winkler in \cite{Lovaszwinkler}.   Doyle won a \$ 50 prize for providing  an `intuitively plausible' explanation for why $K^i$ is constant  in \cite{Doyle_Kemeny},  and Levene and Loizou \cite{Leveneloizou} give a physical interpretation of Kemeny's constant as the expected length of time it will take a random (web) surfer to reach the final destination web-page.    

Given the suggestive notation we're using here to describe these two regularized traces on Markov chains and on surfaces,  one might guess that there should be a connection between the regularized trace on surfaces, $Tr'(\Delta^{-1})=\sum_{j\geq 1}' \frac{1}{\lambda_j}$, and Kemeny's constant on Markov chains , $Tr'(I-P^{-1})=\sum_{j=1}^{n-1} \frac{1}{1-\nu_j}$.  Indeed, the hide-and-seek games give us a way to related these two quantities, by considering the  the expected durations of the  two games.

 Our main result is the following theorem, which  is a paraphrase of Theorem \ref{thm.hideandseek}, and Corollaries \ref{cor.expdurgameI} and \ref{cor.expdurgameII} . 
 \begin{theorem}\label{thm.hideandseekintro} Consider the two hide-and-seek games described above  on either a surface $(M, g)$ with volume $1$, or on a Markov Chain with states $\mathcal{S}$ and transition matrix $P$.
 \ben
 \item In both settings,  the expected duration for Game I, which starts deterministically at point $x\in M$ or state $i\in \mathcal{S}$ , is independent of the point $x\in M$ and $i\in \mathcal{S}$.  Moreover, on a Markov chain, the expected duration of the game is the spectral invariant Kemeny's constant, where $K=`Tr((I-P)^{-1})= \sum_{j=1}^{n-1} \frac{1}{1-\nu_j}$.  On surfaces, the expected duration of the game is essentially the regularized trace, up to a universal  correction depending on $\epsilon$.
\item    The expected duration for Game II, which ends deterministically at a point $y$ or a state $j$  is a density for the corresponding quantities in Game I, and in general, the expected duration of Game II does depend on the deterministically chosen data.
\een
\end{theorem}
The density for the two regularized traces on surfaces and Markov chains may be given as a mass-like term in the Green's function, so Theorem \ref{thm.hideandseekintro} gives a probabilistic interpretation of the regularized trace on surfaces, as well as a probabilistic interpretation of the mass-like quantity in the Green's function.  In order to prove Theorem \ref{thm.hideandseekintro}, we will relate the expected hitting times to the Green's function for $\Delta$ on surfaces, and we will consider the Green's function for the operator $(I-P)$ on Markov chains.  

\paragraph{Organization}
In Section \ref{sec.background}, the relevant background for Markov chains and surfaces is detailed.    Throughout our exposition, we emphasize the analogous statements on Markov chains and on surfaces.  In each case, we consider an operator that gives the difference between the value of a function at a point, and a weighted average value of the function on it's neighbors, which is $(I-P)$ and $\Delta$.  We discuss the spectral resolutions and Green's functions for the relevant operators in each of the two settings.  While we are largely recording well-known background in Section \ref{sec.background}, we believe that the carefully laid out analogies do not appear explicitly in the literature, it seems worthwhile to carefully develop the analogies between the two settings.  In Section \ref{sec.hideseekmarkov}, we analyze the expected duration of the hide-and-seek games on Markov chains, and in Section \ref{sec.hideseeksurface} we consider the expected duration for surfaces.  In each of these two sections, we give a formulation of the expected first hitting times in terms of Green's functions, which we apply in computing the expected duration of each of the hide-and-seek games.    The main results are in Theorem \ref{thm.hideandseek}, and Corollaries \ref{cor.expdurgameI} and \ref{cor.expdurgameII}.

\section{Background}\label{sec.background}
In this section we give the necessary background in order to set-up our hide-and-seek games and their probabilistic interpretations on Markov chains and on surfaces.  We draw parallels between the analogous statements for the two settings throughout this discussion.
\subsection{Markov Chain background}
A  discrete,  time-homogeneous  Markov chain is a sequence of random variables, $X_0, X_1, X_2,\ldots$, valued in   $S=\{1,2,3,\ldots n\}$.  From the time-homogeneity and Markov properties, we know that the Markov chain is  characterized by initial distribution, $p_0$, and  transition matrix $P$ with $(P)_{ij}=P(i,j)\geq 0$  where:
\be
(P)_{ij}=P(i, j)=\mathbf{P}(X_k=j| X_0=i_0, X_1=i_1,\ldots X_{k-1}=i)=\mathbf{P}(X_k=j|  X_{k-1}=i)
\ee
In words, the $ij$ entry of the transition matrix, $P_{ij}$ gives the probability that the Markov chain will go from state $i$ to state $j$ at any given time step.   

\paragraph{Difference and averaging operator,  $(I-P)$}
We will consider the operator $I-P$, which gives the difference between the value of the function at a point, and the weighted average of the values at the neighboring states (here, a state $j$ is a neighbor if $P_{ij}\neq 0$). Consider the vector $\mathbf{f}=(f_j)_{j=1}^n$ to be a function, and compare the following equation to the analogous statement for the Laplacian on surfaces, which appears in \eqref{e.laplaceisaverage}  in the next section.
\begin{equation}\label{e.iminuspisaverage}
[(I-P) f]_i=f_i - \sum_{j}P_{ij} f_j
\end{equation}
In terms of electrical sources and potentials,  the operator $(I-P)$ maps potentials to source densities.\footnote{A careful treatment of the  connections between Markov Chains and electrical networks and potential theory is given by Doyle and Snell in \cite{RWEN} }
\paragraph{The spectrum of $(I-P)$}  
The matrix $P$ is row stochastic, so  $ \sum_{j=1}^{n} P_{ij}=1$ for all $ 1\leq i \leq n$, and the constant vector $\mathbf{1}$ is a right eigenvector for $P$, with eigenvalue $1$:
\be
P \mathbf{1}=\mathbf{1}
\ee
From linear algebra, we know that there will also be a right eigenvector for $P^T$ corresponding to the eigenvalue $1$, where $P^T w=w$, and taking the transpose will give rise to $w^T P=w^T$.  
\begin{definition} If the Markov chain is irreducible and ergodic, then the eigenvalue $1$ has multiplicity $1$, and there is a unique positive left eigen-measure, $w$,  which is  called the {\it equilibrium measure} and satisfies  
\be
w^T P = w^T , \hspace{10pt}  w_i>0, \hspace{10pt}\sum_i w_i =1
\ee
\end{definition}
We recall two interpretations of the equilibrium measure that we will exploit in the sequel:
\begin{prop}[See Grinstead and Snell, Ch. 11 \cite{Grinsteadsnell} ]\label{p.interpretw} For an irreducible, ergodic Markov chain $P$ with equilibrium measure $w^T$, we have 
\ben
\item  For large $n$, $w_j$ is the probability that $X_n=j$ (by the law of large numbers), and.
\begin{equation}  
P^n\rightarrow W \text{    as   } n\rightarrow \infty \text{  where  }  W=(W_i), \text{ with } W_i=(w^T) \end{equation}
\item  The expected first return time for $j$ is $\frac{1}{w_j}$.  
\een
\end{prop}
Also, if $P$ is irreducible and ergodic, then $P$ has eigenvalues $1>\nu_1\geq \nu_2 \cdots \nu_{n-1}>-1$. 
 
 \paragraph{Spectrum of $(I-P)$.}  Now, we will be interested in the operator $(I-P)$, which has eigenvalues $0, 1-\nu_1, 1-\nu_2\ldots 1-\nu_{n-1}$.  For an irreducible ergodic Markov chain, we know that the operator $(I-P)$  has a $1$-dimensional null-space, and its range is $(n-1)$ dimensional:
\begin{equation}\label{e.nullrangeiminusp}
(I-P)\1 = 0; \hspace{30pt}\text{ and  } \w^t (I-P)=0, \hspace{5pt} \text{ so } \forall v\in \mathbb{R}^{n}, \w^t(I-P)v=0 \hspace{10pt}\forall v\in \mathbb{R}^{n}
\end{equation}
Thinking electrically, the fact that $\mathbf{1}$ is in the null-space says that  if we start with a potential that is identical at each state, there will be no net current flowing anywhere at all, and so the source strength at each state will be $0$.  The second statement says that the source density corresponding to any potential $v$ must have total source strength $0$, where the source strength is measured with respect to the equilibrium measure $w$.  

\paragraph{The Green's function and inverting $(I-P)$}
We are interested in inverting the operator $(I-P)$, but now we can see that we can only hope to invert the operator on its range, namely  the $(n-1)$ dimensional subspace, $\{v \in \rn: v \perp w\}$.   Since $\mathbf{1}$ is in the null-space for $(I-P)$, the inverse will only be unique up to a choice of an additive constant.  In order to specify the Green's function,  which is given by the  matrix $G=(G_{ij})$,  we take a frame  of $n$ source densities, $S_j$ with  $S_j\in \rn\perp w$ corresponding to  having a source of strength $1$ at $j$, and a distributed sink of total  strength $1$.  The source density, $S_j$, is the analog to the source density on surfaces which is $\delta_x(y) -\frac{1}{V}$.  On Markov chains, $S_j$ may be  given in terms of the standard Euclidean basis vector $\mathbf{e}_j$ by:
\begin{equation}\label{e.deltaformarkov}
\mathbf{S}_j=\frac{1}{w_j} \mathbf{e_j} -\mathbf{1}, \text{   which satsifies   } \sum w_i S_j(i)=w_j\frac{1}{w_j} -\sum w_i = 0.
\end{equation}
 If we can find the potential corresponding to each of the source densities,  $S_j$, then we can  solve the equation $(I-P)u=f$ as long as $f\perp w$.   For each $j$, we would like to solve the following equation, which should be compared with the distributional equation for the Green's function on surfaces, given below in \eqref{e.distreqngxy}:
 \begin{equation}\label{e.disteqnforGij}
 [(I-P)G]_j = \mathbf{S}_j=\frac{1}{w_j} \mathbf{e_j} -\mathbf{1}    \text{  or  } (I-P)G=(\mathbf{S}_1 \mathbf{S}_2 \cdots \mathbf{S}_n)
  \end{equation}
 Since the constants are in the null-space for $(I-P)$, the Green's function will only be unique up to a choice of an additive constant.  We will   demand that $WG=0$, and Proposition \ref{p.interpretz}  will demonstrate that our choice  is probabilistically natural.  Moreover, this choice of normalization is the analogous choice to taking $\int G(x,y) \dif V(x)=0$ on surfaces, which ensures that $\Delta^{-1}: 1\rightarrow 0$.  In the Markov chain setting, $G$, the Green's function  exists, and the following Proposition gives a  useful concrete representation for the Green's function.
 \begin{prop}\label{p.greensfnonmarkov}[Meyer, \cite{Meyer}]  Consider  an irreducible and ergodic  Markov chain with transition matrix $P$, and (unique positive) equilibrium measure $w^T$, and let $W$ denote the rank one matrix with rows $W_j=w^T$, and consider the diagonal  matrix $D=Diag(\frac{1}{w_i})$  with entries $D_{ii}=\frac{1}{w_i}$, and $\mathbf{1}$ denote the $n\times n$ matrix where each entry is a $1$.  There  is a unique Green's function, denoted by $G=(G_{ij})$ which satisfies the following two requirements:
 \begin{equation}\left\{ \begin{array}{ll} 
 (I-P)G=D-\mathbf{1} & \text{  or } (I-P)G=(I-W)D  \label{e.inverseiminusp}\\
 WG=0 & 
 \end{array} \right.
\end{equation}
 The Green's function is given by:
 \begin{equation}\label{e.greensforp}
 G=\left( \left( I-(P-W)\right)^{-1} - W\right) D  
 \end{equation}
\end{prop}
\begin{rmk}
The matrix $D$ may be thought of as playing the same role as the inverse metric plays in Riemannian geometry.  The inverse metric, raises an index by turning a covector into a  vector.    In the Markov chain setting, a co-vector is a (possibly signed) measure, and by acting on the right, the matrix $D$ turns a measure into a potential, which is a vector.   The following interpretation of the measure was also given by Meyer in \cite{Meyer}.  
\end{rmk}
\begin{prop}\label{p.interpretz}[Compare Theorem 3.2 of Meyer, \cite{Meyer}]  The normalization that we have chosen for the Green's function, given by $WG=0$, gives rise to a signed measure obtained from the Green's function:
\begin{equation}\label{e.definez}
Z=GD^{-1}=\left( \left( I-(P-W)\right)^{-1} - W\right)
\end{equation}
The expected number of excess visits to a state $j$, for a Markov chain with an initial distribution of $p_0=\mathbf{e}_i$, versus a Markov chain with  an initial probability distribution given by $p_o=w$ is given by $Z_{ij}$.   \end{prop}
\begin{rmk}  Our notation for $Z=\left( I-(P-W)\right)^{-1} - W$ as the fundamental matrix  is consistent with Aldous' notation in \cite{Aldousfill}, however,  Grinstead and Snell in \cite{Grinsteadsnell} use a different normalization and their fundamental matrix is just $\left( I-(P-W)\right)^{-1} $.
\end{rmk}
\begin{rmk}  Speaking informally, if you spend a little extra time in a state $j_1$, you won't be able to spend quite as much time in the other states, so  the total number of excess visits to all of the states must be $0$.  Therefore, the probabilistic interpretation of Proposition \ref{p.interpretz} makes it clear that  $\sum_{j} Z_{ij}=0$.   
\end{rmk}
\begin{proof}

In order to verify the probabilistic interpretation of $Z_{ij}$, we need to consider two Markov processes, both governed by the original transition matrix $P$, but  having two different initial distributions:   $\mathcal{P}_{i\bullet}$ corresponds to having the initial distribution $p_o=\mathbf{e}_i$, which means the chain is started at state $i$ with probability $1$, and  $\mathcal{P}_{w\bullet}$ corresponds to having the initial distribution $p_o=w$, which means the chain is started at random with respect to the equilibrium measure $w$.  So, $\mathcal{P}_{ij}=P_{ij}$, which is the probability that the chain goes from state $i$ to $j$ in one step, and since $(w^TP)_j=w_j$, we know that $\mathcal{P}_{wj}=w_j$, which is the probabillity that the state will be in state $j$ after one step, given that it started at a random state chosen with respect to the equilibrium measure.  

Now, we introduce  a $0-1$ random variable $\chi^k_{ \alpha j}$, where  $\alpha$ is the initial distribution of the Markov chain, and 
\be
\chi^k_{ \alpha j}=\left\{ \begin{array}{ll}
1& X_k=j \text{ given } X_0 \text{ determined by  } p_0=\alpha \\
0 &  \text{ else }
\end{array}\right.
\ee
Now, since $P_{ij}^k$ is the probability that the Markov chain is in state $j$ at step $k$, given that it started in state $i$ (\ie given that $\alpha =\mathbf{e}_i$), we know that $\chi^k_{ \mathbf{e}_ij}=1$ with probability $(P^k)_{ij}$.  Since $w^T P^k=w^T$, we know that $\chi^k_{ wj}=1$ with probability $w_j$.  The total number of visits within the first $N$ steps for a process with initial distribution $\alpha $  would be $\sum_{k=1}^N  \chi^k_{ \alpha  j}$, so the expected excess number of visits to a site $j$ within the first $N$ steps for the two processes is given by: 
\begin{multline*}
\mathbb{E}[\sum_{k=0}^N \chi^k_{ \mathbf{e}_ij}-\sum_{k=0}^N \chi^k_{ wj}]=1\cdot  \delta_{ij}+1\cdot P_{ij}+1\cdot P^2_{ij}+\cdots1\cdot  P^N_{ij}-\stackrel{n}{\overbrace{\left(1\cdot w_j +1\cdot  w_j + \ldots + 1\cdot  w_j\right)}}\\= \delta_{ij}+P_{ij}+P^2_{ij}+\cdots P^N_{ij}-(N+1)w_j
\end{multline*}
Now, since this expression holds for any $j$, we may write it in terms of matrices, and since the sum is finite, we may rearrange the sum as follows:
\be
\mathbb{E}[\sum_{k=0}^N \chi^k_{ \mathbf{e}_i}-\sum_{k=0}^N \chi^k_{ w}]=[I-W+(P-W)+(P^2-W)+\cdots (P^N-W)]_{ij}
\ee
Now, by using the fact that $w$ is the stationary measure for $P$ and the related fact that  $PW=WP=W$, if  coefficients are  collected, one can show that $(P^k-W)=(P-W)^k$.  Now, we may re-write our expression for the expected excess visits within the first $N$ steps, and we can recognize a truncation of the Neumann series for the fundamental matrix, $Z$, defined in \eqref{e.definez}:
\bea
\mathbb{E}[\sum_{k=0}^N \chi^k_{ \mathbf{e}_i}-\sum_{k=0}^N \chi^k_{ w}]&=&[I-W+(P-W)+(P^2-W)+\cdots (P^N-W)]_{ij}\\
&=& [I-W+(P-W)+(P-W)^2+\cdots (P-W)^N]_{ij}\\
&=& \left[\left( \sum_{l=0}^N (P-W)^k \right)-W\right]_{ij}\\
\eea
Of course,  in order to compute the expected number of excess visits between the two processes, we must sum over all $l$, and so we have
\be
 \mathbb{E}[\text{ Excess visits to }j \text{ for  }\mathcal{P}_{ij}\text{ versus }\mathcal{P}_{wj}]=\left[\left( \sum_{l=0}^\infty (P-W)^k \right)-W\right]_{ij}=Z_{ij}
 \ee
 In case the reader is worrying about issues of convergence for this infinite series, we should observe here that the Law of Large Numbers for irreducible and ergodic Markov chains implies that $P^n\rightarrow W$, and since the other  eigenvalues for $P$ will be less than $1$, we can estimate $|P^n-W|\leq |\lambda_2|^n$, where $|\lambda_2|<1$, and  so  the Neumann series for $Z$ converges.
 \end{proof}

\subsection{Surface background}
The second story presented in this paper takes place on a  compact surface without boundary $(M,g)$, so we now introduce the same cast of characters that was just introduced in the context of Markov chains.     In order to set-up the notation, let  $\dif V$ denote the volume element, $\nabla$ denote the gradient, and  $\Delta = -div\cdot \nabla$ denotes the positive Laplacian operator corresponding to $g$.  If $(M,g)$ has volume $V$, then the corresponding probability measure will be $\frac{\dif V}{V}$. 
\paragraph{Difference and averaging operator, $\Delta$}
The operator $\Delta$ is analogous to the operator $(I-P)$ on Markov chains, since both operators give the difference between the value of the function at a point and the weighted average over  neighboring points.  In the surface setting, we recall that the Laplacian operator  gives the $\epsilon^2$ difference between a function $f$ at a point $x_0$, and the average of the neighbors on a sphere of radius $\epsilon$ by considering Taylor's theorem and writing the metric in normal coordinates at the point $x_0$.  This statement should be familiar from the behavior in $\mathbb{R}^n$, where the most familiar evidence of this averaging behavior is the  well-known `Mean Value Property' for a harmonic function: if $\Delta f = 0$ then  the value of the function at a point is equal to the average value on a sphere of radius $\epsilon$.  Note that we can compare the surface formulation of this averaging property to that of $(I-P)$, which was stated in  \eqref{e.iminuspisaverage}
:
\begin{equation}\label{e.laplaceisaverage}
 \Delta f(x_0) \epsilon^2 \approx f(x_0)-\frac{1}{Vol(S_\epsilon^{1})}\int_{S_\epsilon^{1}(x_0)} f(x) \dif S \text{  for  } f\in C^\infty(M)
\end{equation}

\paragraph{Spectrum of $\Delta$ and the regularized trace}
Since the Laplacian is a self-adjoint elliptic operator of order $2$, it has discrete eigenvalues $\{\lambda_j\}_{j=0}^\infty$ that 
  are bounded
below, with $0<\lambda_1\leq \lambda_2 \cdots$, and Weyl's law
describes the growth of the eigenvalues:
\begin{equation}\label{e.weylslaw}
\text{  Weyl's Law:  }\hspace{5pt}\lim_{j\rightarrow \infty} \frac{\lambda_j}{j}=\frac{4\pi}{Vol(M)} 
\end{equation}
  Now, in order to make  connections with a spectral invariant precise, we introduce   the operator $\Delta^{-s}$, which is  a linear operator on   $L^2(M, g)$.  A branch of the logarithm is fixed  so that $\lambda^s=e^{s\log \lambda }$, with $1^s=1$.  In order to define the regularized trace, we introduce the operator $\Delta^{-s}$.  From Weyl's law and the functional calculus, we see that  $\Delta^{-s}$ is a Hilbert-Schmidt operator for $\Re s\geq 1$.  If $\Re(s)>1$, then $\Delta^{-s}$ is trace-class, and the spectral Zeta function is defined to be the trace of $\Delta^{-s}$:
\begin{equation*}
Z(s)\stackrel{def}{=}Trace(\Delta^{-s})=\sum_{j=1}^{\infty} \lambda_j^{-s} \hspace{30pt} \Re s>1
\end{equation*}
The zeta function has  a simple pole at $s=1$, and it is well known that the zeta function may be continued to a meromorphic function defined 
 on the complex plane. In
\cite{Morpurgo}, Morpurgo considers the following spectral
invariant, which is  the regularized trace of  $\Delta^{-1}$:
\begin{definition} \label{d.regzeta} Assume $(M,g)$ is a closed compact surface with volume normalized to be $4\pi$, then the
{\it regularized trace of $\Delta$}, or the {\it regularized Zeta function at 1} is defined as follows:
\begin{equation}\label{e.defregzeta}
\tilde{Z}_g(1)=\lim_{s\rightarrow 1}
\left(Z(s)-\frac{1}{s-1}\right)
\end{equation}
\end{definition}
The regularized trace may also be approached via the Green's function for $\Delta$, as will be stated in Proposition \ref{p.massisspecdensity}.

\paragraph{The Green's function and $\Delta^{-1}$} As is the case for Markov chains, the Green's function allows us to invert the Laplacian, $\Delta$, a linear operator on the Sobolev space $H_2(M)\subset L^2(M,g)$.    The null-space and range are $\Delta$ are characterized by the following analog to the Markov chain situation described in \eqref{e.nullrangeiminusp}.  
\begin{equation}\label{e.nullrangedelta}
\Delta 1=0; \hspace{30pt} \int_M \Delta f \dif V=0 \forall f\in H_2(M)
\end{equation}

Again, the inverse can not be defined on all of $L^2(M, g)$, since the image of $\Delta$ consists of the functions that are orthogonal to the constants.  Additionally,  the inverse for $\Delta$ will only be unique up to the choice of a constant.  
As was the case on Markov chains, the Green's function, $G(x,y)$  is the potential corresponding to a source density that has a source strength of $1$ at a point $x$ and a distributed sink that has a total strength of $1$, which corresponds to solving the distributional equation \eqref{e.distreqngxy}.  We will choose the constant so that the $\int G(x,y)\dif V(x)=0$, or in otherwords $\Delta^{-1}$ takes constant functions to $0$.  The following proposition describes the Green's function on a surface.  While we cannnot write down the Green's function explicitly on a surface in the way that we wrote down the Green's function on a Markov chain in Proposition \ref{p.greensfnonmarkov}, the parametrics in \eqref{e.parametricsgf} is the next best thing.  Since the surface behaves locally like a piece of the plane, we expect to see the same logarithmic singularity that one sees for the Green's function on Euclidean space.  
\begin{prop}\label{p.greensfnonsurface}[See, for example, Aubin \cite{Aubin}] On a surface $(M, g)$ with volume $V$, there exists a unique Green's function, $G(x,y)$ which is the  integral kernel for the operator $\Delta^{-1}$, it is symmetric in $x$ and $y$, and satisfies the following two equations:
\begin{eqnarray}
&&\Delta_y^{distr} G(x,y)\stackrel{\dif V}{=} \delta_x(y) -\frac{1}{Vol(M)} \label{e.distreqngxy}\\
&&\int_M G(x,y) \dif V(y)=\int_M G(x,y) \dif V(x) = 0  
\label{e.normalizegreensfn} 
\end{eqnarray}
The Green's function is smooth away from the diagonal, and for small geodesic distance, $d(x,y)$, there is a logarithmic singularity, with:
\begin{equation}\label{e.parametricsgf}
G(x,y)=\frac{-1}{2\pi} \log(d(x,y))+m(x)+O(d(x,y)) 
\end{equation}
The term $m(x)=\lim_{y\rightarrow x} [G(x,y)+\frac{1}{2\pi} \log d(x,y)]$ is called the {\it Robin's mass}.
\end{prop}
\paragraph{Comparing the two methods of regularizing the trace for $\Delta$}
 In the context of this paper, we will be  considering expected hitting times, which means we will be working with the Green's function, so it is worth recording the fact that the regularized trace may either be defined through the zeta function as in \eqref{e.defregzeta}, or by integrating the Robin's mass, which appears in \eqref{e.parametricsgf}.    The following proposition gives the precise relation between the two quantities:
\begin{prop}\label{p.massisspecdensity} [Morpurgo \cite{Morpurgo_Duke}, S. \cite{steiner_duke}]
Assume $(M,g)$ is a closed compact surface with volume $V$.  Then, the Robin's
mass $m(x)$ appearing in \eqref{e.parametricsgf} is a density for the spectral invariant
$\tilde{Z}(1)$ defined in \eqref{e.defregzeta}, and the regularized trace may be written in terms of the integral of the Robin's mass as follows:  
\begin{equation}\label{e.massisspecdensity}
\tilde{Z}(1)=\int_M  m (x) \dif V(x) +\frac{V}{4\pi}(-2\log 2+2\gamma)  \text{  with Euler gamma:  } \gamma\approx .5772
\end{equation}
\end{prop}

\section{Hide-and-seek and Kemeny's Constant on Markov Chains}\label{sec.hideseekmarkov}
\paragraph{Expected first hitting times via the Green's function}
In order to analyze the expected duration of the hide-and-seek games on a Markov chain that were explained in Items \ref{i.gameI} and \ref{i.gameII} of the Introduction, we need to know what the expected first hitting times are to get from a state $i$ to a state $j$.  After defining the hitting times, for expository purposes, we will formulate the system of equations that the hitting times must satisfy, we will then compute these hitting times in terms of the Green's function.  A good reference for the relation between the expected first hitting times and the Green's function on Markov chains is the Grinstead and Snell's text on Probability, \cite{Grinsteadsnell}.
 \begin{definition}\label{def.exphitonmarkov}  Let $m_{ij}$ denote the {\it expected first hitting time} for a   Markov process to reach state $j$, given that it started in state $i$.  In other words, if $T$ is the random variable satisfying  $T=\inf\{n\geq 1 : X_n=j\}$, then
  \begin{equation}\label{e.deffirsthitmarkov}
  m_{ij}\stackrel{def}{=}\mathbb{E}[T|X_0=i] ,   \text{   the expected first hitting time.  \hspace{10pt} Take }m_{ii}=0
  \end{equation}
\end{definition}
Now, in order to see the relation between the Green's function and the first hitting times, we derive  \eqref{e.hittimeineqj}, which is satisfied by  the expected hitting time $m_{ij}$.  We  use a `first-step' analysis, \ie we will exploit the adage  that states, `A journey of a thousand miles begins with a single step':  if you are wandering on the Markov chain at random and you find yourself at  a state $i$, then, you will step to any state $k$ with probability $p_{ik}$, and the expected length of the remaining  journey to state $j$  will be the sum of that first step and the expected time it will take you to then get from $k$ to $j$, \ie  
\be
m_{ij}=1+\sum_{k}p_{ik}m_{kj} 
\ee
Consequently, the matrix of expected hitting times must satisfy:
\begin{eqnarray}
 [(I-P)M]_{ij}&=& 1 \text{ for }j\neq i  \label{e.hittimeineqj}  \\
m_{ii}&=& 0 \nonumber
\end{eqnarray}
By a similar first-step analysis for a return trip to $i$, one sees that we must have $1+\sum_{k=1}^n p_{ik} m_{ki}=\frac{1}{w_i}$, since if we take one step from state $i$, we will go to any other state $k$ with probability $k$, and then we expect the remainder of the journey to take $m_{ki}$.  Of course, from Proposition \ref{p.interpretw}, we know that the expected return time to $i$ is exactly $\frac{1}{w_i}$.     Consolidating all of the requirements on $m_{ij}$, we find that the expected hitting times must solve the following equations: 
\begin{eqnarray}\label{e.markovhittingtimes}
(I-P)M&=&\ones - (\frac{1}{w_i} e_i)=\ones-D\\
m_{ii} &=& 0 \nonumber 
\end{eqnarray}
Now, these equations should look familiar to us from the discussion of the Green's function in Section \ref{sec.background}.
\begin{lemma}\label{l.hittimegreensmarkov}[Grinstead and Snell \cite{Grinsteadsnell}]  For an irreducible ergodic  Markov chain where $G$ denotes the Green's function  defined in Proposition \ref{p.greensfnonmarkov}, and $M$ denotes the matrix of expected hitting times defined in \ref{def.exphitonmarkov}, then $M$ may be represented in terms of the Green's function:
\begin{equation}\label{e.hittimeingonmarkov}
m_{ij}=G_{jj} -G_{ij} 
\end{equation}
\end{lemma}

\paragraph{Kemeny's constant and density}
Now, we are finally in a position to analyze the expected duration of the hide-and-seek games on a Markov chain.
\paragraph{Game I}
Observe that for the first hide-and-seek game, in which a seeker begins at a deterministically chosen state $i$ and  must find a hider located at a randomly chosen state $j$  chosen with probability $w_j$, the expected duration of the game is given by the following quantity:
\begin{equation}\label{e.kemconst}
K^i\stackrel{def}{=}\sum_j m_{ij} w_j
\end{equation}
{\it A priori}, the expected duration of the hide-and-seek game should depend on the starting point $i$, as indicated by the $i$ dependence in  \eqref{e.kemconst}.
\begin{theorem}[Kemeny, \cite{Kemenysnell}]  The quantity $K_i$ is a constant $K$, independent of $i$.  \end{theorem}
We will refer to the constant  quantity, $K$ defined in Equation \ref{e.kemconst}  as {\it Kemeny's constant}. 

One can prove this theorem by showing that $K_i$ is harmonic, and therefore constant, as is done by Doyle  in \cite{Doyle_Kemeny}. Or, one can prove the theorem explicitly computing $K_i$ in terms of the Green's function.  By  substituting the expression for $m_{ij}$ that was given  in terms of the Green's function in \eqref{e.hittimeingonmarkov}, and observing that one is left with a quantity that is independent of $i$, namely, for the matrix $Z$ addressed in Proposition \ref{p.interpretz}, which is the signed measure corresponding to $G$, one finds:
\be
K=K^i=\sum_j Z_{jj}=Tr(Z)
\ee
 Of course, the trace of a matrix is a spectral invariant, which is the content of the next theorem.
 
  \begin{theorem}[Levene-Loizou, \cite{Leveneloizou}, Aldous \cite{Aldousfill}, Meyer \cite{Meyer}]  Kemeny's constant is a spectral invariant.  If the eigenvalues for $P$ are given by $1, \nu_1, \ldots \nu_{n-1}$, then,
\be
K=K^i=Tr(Z)=\sum_{j=1}^{n-1} \frac{1}{1-\nu_j}.
\ee
\end{theorem}
Heuristically, Kemeny's constant can be thought of as $Tr((I-P)^{-1})$, but of course the honest approach is to use the expression for the Green's function given in \eqref{e.greensforp} and observe that $K=Tr(Z)=Tr((I-(P-W))^{-1}-W)$, and $Z$ has eigenvalues $0, 1-\nu_1, 1-\nu_2, \ldots 1-\nu_{n-1}$.

Now, we may give the following interpretation of the Game I, in which the seeker's state  is chosen deterministically at a state $i$, and the  hider's state is chosen at random.
\begin{cor}\label{cor.expdurgameI}  On a Markov Chain, the expected duration of Game I, detailed in Item \ref{i.gameI}  of the Introduction, is a spectral invariant that is independent of the starting point $i$. 
\end{cor}
\paragraph{Game II}
Now, in Game II, the seeker is located at a randomly chosen state, $i$ with probability $w_i$, and a determinstically located hider  must be found at $j$.  The expected duration of this game is given by:
\begin{equation}\label{e.pKnc}
K_j=\sum_i w_i m_{ij}
\end{equation}
\begin{lemma}  The quantity $K_j$ is a density for Kemeny's constant, and it is the diagonal  term of $G$, the Green's function defined in \ref{p.greensfnonmarkov}:
\be
K_j=G_{jj}
\ee
\end{lemma}
Now, we may connect this density to our second hide-and-seek Game:
\begin{cor}\label{cor.expdurgameII}  The expected duration of hide-and-seek Game II, detailed in Item \ref{i.gameII}, is given by  the quantity $K_j$ of \eqref{e.pKnc}, and in general it will depend on the point $j$.  The expected duration of Game II is a density for Kemeny's constant (here the measure is  the stationary measure, $w$), and thus $K_j$  is a density for a spectral invariant.
\end{cor}

\section{Hide-and-seek and the regularized trace on Surfaces}\label{sec.hideseeksurface}
Now, we would like to consider the same two hide-and-seek Games on a surface, but we must begin by discussing a stochastic process.  The seeker's path will follow the laws for Brownian motion.  

 \paragraph{Brownian Motion on manifolds and expected first hitting times via Green's functions} (See \cite{Hsu} for details concerning the construction of Brownian Motion on a manifold).   In order to compute the expected duration of the hide-and-seek games, we must consider the behavior of a Brownian `seeker'.    On a  surface, $(M, g)$, one can construct a Brownian motion $\{\Sigma_t^x\}_{t\geq 0}$,  which begins at a  point $x\in M$, \ie $\Sigma_0^x=x$.   Since a closed compact surface is stochastically complete, the Brownian motion will be recurrent.  While the probability of hitting any particular point $y$ is $0$ almost everywhere, the probability of hitting the boundary of an $\epsilon$-ball around $y$ will be $1$.  We let  the random variable $\tau^x(B_\epsilon(y))$ denote the first hitting time for a Brownian motion beginning at $x$ to reach an $\epsilon$ ball around $y$, {\it i.e.,} 
 \be
 \tau^x(B_\epsilon(y)) =\inf \{t>0: \Sigma_t^x\in \partial B_\epsilon(y)\}, 
 \ee
 and we write  the expected value of $\tau^x$ for a given epsilon ball  as $H_\epsilon(x,y)=\mathbb{E}[\tau^x (B_\epsilon(y)]$. Recall that $\Delta$ is the analg to the operator $(I-P)$ in the Markov chain setting, so the quantities $H(x,y)$ and $m_{ij}$ are the first hitting times for surfaces and Markov chains respectively, and it should come as no surprise that $H(x,y)$ solves a partial differential equation which is the continuous space analog of \eqref{e.hittimeineqj}.    In particular, it is well-known 
 that $H_\epsilon(x,y)$ solves  the following boundary value problem\footnote{Here we consider a re-scaled Brownian motion to follow the geometer's convention for the Laplacian}:
\begin{equation} \label{e.hittimepde}
 \left\{ \begin{array}{r@{\quad \quad\quad }l}
\Delta H_\epsilon(x,y)=1 & x\in M\setminus B_\epsilon(y)\\ 
H_\epsilon(x,y)=0 & x\in  \partial [M\setminus B_\epsilon(y)] 
\end{array} 
\right.
\end{equation}
Now, the expected hitting time, $H_\epsilon(x,y)$ can be written in terms of the Green's function, as stated in the following lemma, which is the continuous analog to Lemma \ref{l.hittimegreensmarkov}  on Markov chains.  Notice that on surfaces we need to introduce a $\log \epsilon$ regularization, which reflects the fact that it would take infinitely long to actually hit the point $y$.
\begin{lemma}\label{l.hittimegreenssurf}
On a surface $(M,g)$, if $H_\epsilon(x,y)$ is the expected value of the first hitting time for a Brownian motion beginning   at $x$ to reach   $B_\epsilon(y)$  satisfying \eqref{e.hittimepde}, and $G(x,y)$ is the Green's function for the Laplacian characterized in Proposition \ref{p.greensfnonsurface}, then $H_\epsilon (x,y)$ may be written uniquely in terms of $G(x,y)$ and the robin's mass $m(y)$ appearing in \eqref{e.parametricsgf} as follows: 
\begin{equation}\label{e.epsilonhittime}
 \left\{ \begin{array}{r@{\quad \quad\quad }l} 
 H_\epsilon(x,y)=-V G(x,y)-\frac{V}{2\pi} \log \epsilon +V m(y)+u_\epsilon(x) & \hspace{-20pt} x\in M\hspace{-4pt}\setminus\hspace{-4pt} B_\epsilon(y)\\
 H_\epsilon(x,y)=0 &\hspace{-20pt}  x\in \bar{B_\epsilon(y)}
 \end{array}
 \right.
 \end{equation}
Here,   $u^\epsilon(x)$  is  harmonic on $M\setminus{ B_\epsilon(y)}$ and  $\|u^\epsilon\|_{C^0}=O(\epsilon)$.  
  \end{lemma}
  \begin{proof}
  While one only needs to verify that $H_\epsilon(x,y)$ defined in Equation \ref{e.epsilonhittime} satisfies \eqref{e.hittimepde}, we believe it is more illuminating to prove the lemma by demonstrating how to construct the correct function, $H_\epsilon(x,y)$.
We  begin by observing that a good candidate for solving the  hitting time problem in \eqref{e.hittimepde} would be $- G(x,y)$, since we know from \eqref{e.distreqngxy} that if $x\neq y$, we have (for $V=Vol(M, g)$) 
  \be
  \Delta_x^{distr} G(x,y)\stackrel{\dif V}{=}-\frac{1}{V} \text{  for } x\neq y
  \ee
Of course, away from $y$, the right-hand side of the equation is smooth, so it holds point-wise for $x\neq  y$, and we must re-scale the equation so that we are working with the probability measure, which is given by $\frac{\dif V}{Vol(M)}$:
  \be
  \Delta_y\left(-Vol(M) G(x,y)\right) \stackrel{\frac{\dif V}{Vol(M)}}{=} 1  \text{  on } M\setminus \overline{B_{\epsilon}(y)}
  \ee
  Now, while we have arranged for  the Laplacian to behave  properly on $M\setminus \overline{B_{\epsilon}(y)}$,   but of course, we need to normalize our candidate it so that the hitting time is $0$ at the boundary, so we will have  $H^\epsilon_y(x)=-Vol(M)G(x,y)+(correction)$.  From the parametrics of the Green's function given in \eqref{e.parametricsgf}, we know that for $x\in \partial B_\epsilon(y)$, we have:
  \be
  -Vol(M) G(x,y)=\frac{V}{2\pi} \log d(x, y) -Vm(y)+O(d(x, y))=\frac{V}{2\pi} \log \epsilon -Vm(y)+O(\epsilon)
  \ee
  Now,  by correcting $-Vol(M)G(x,y)$ with the logarithmic term and the Robin's mass term, we have the hitting time that solves the correct differential equation on the interior of our region, and  it will be  correct up to a function that is $O(\epsilon)$ on the boundary of the ball.  In order to obtain an exact solution we need to deal with the $O(\epsilon)$ correction term, without affecting the differential equation on the interior.  In particular, we must find a function $u_\epsilon(x)$ satisfying:
  \bea
u_\epsilon(x)=  -Vol(M) G(x,y)-(\frac{V}{2\pi} \log \epsilon -Vm(y)) && x\in \partial B_\epsilon(y)\\
\Delta u_\epsilon(x)=0 &&  x\in M\setminus \overline{B_\epsilon(y)}
\eea
From standard PDE theory on a surface (see, for example, \cite{Taylor1}), it is known that we can solve this equation for the harmonic extension $u_\epsilon(x)$, and $u_\epsilon(x)$ will be unique.  From the maximum principle, we know that $u_\epsilon$ will attain its maximum on the boundary, $ \partial B_\epsilon(y)$, so we know that
\be
\|u_\epsilon\|_{C^0(M\setminus \overline{B_\epsilon(y)}}=O( \epsilon)
\ee
Now, since $\Delta u_\epsilon(x)=0$, we may add it to the modified Green's function without disrupting the differential equation on the interior, so we now know that the function $H_\epsilon(x,y)$ defined in Equation \ref{e.epsilonhittime} does satisfy the Dirichlet problem for the expected hitting times.  The function $H_\epsilon(x,y)$ is unique and positive because $H_{\epsilon}(x,y)$ is sub-harmonic  and $H_\epsilon (x,y)=0$ on the boundary, so by the maximum principle, since it is not constant,  it must be positive on  $M\setminus \overline{B_{\epsilon}(y)}$.
  \end{proof}
  Now that we know the expected hitting time for an individual point, we may consider the two hide-and-seek games.
  \paragraph{Hide-and-seek on surfaces}
  We summarize the behavior of the two hide-and-seek games in the following theorem.
  
 \begin{theorem}\label{thm.hideandseek} On a surface $(M, g)$ with volume $V$ and Robin's mass $m(x)$  defined in  \eqref{e.parametricsgf},  the two hide-and-seek games described in the Introduction exhibit the following behavior on $M$.
 \ben
 \item The expected duration for Game I, which starts deterministically at $x$, is \\ 
 \be 
\int_{M\setminus B_\epsilon(x)} H_{\epsilon}(x,y) \dif V(y) =-\frac{V}{2\pi} \log \epsilon +\int_M m(y)\dif V(y)+O(\epsilon),
\ee
 and the expected duration for Game I  is independent of  the starting point, $x$.
\item    The expected duration for Game II, which ends deterministically at $y$  is  \\ $\int_{M\setminus B_\epsilon(y)} H_{\epsilon}(x,y) \dif V(x)  =-\frac{V}{2\pi} \log \epsilon +V m(y)+O(\epsilon)$.
\een
\end{theorem}
\begin{proof}
In order to verify the first statement, we must allow the point $y$ to be randomly chosen with respect to the probability measure $\frac{\dif V(y)}{V}$,  so we will integrate the $\epsilon$-regularized hitting time, $H_\epsilon(x,y)$ given in Lemma \ref{l.hittimegreenssurf} with respect to $y$.  Note that if according to the rules of our game, we have already found any point $y\in B_\epsilon(x)$, so we need only integrate over $M\setminus B_\epsilon(x)$:
 \begin{multline}
\mathbb{E}[\text{duration Game I}]=\int\limits_M H^{\epsilon}(x, y)\frac{ \dif V(y)}{V}\\
 =\int\limits_{M\setminus B(x, \epsilon)} H^\epsilon(x,y)\frac{\dif V(y)}{V}=\int\limits_{M\setminus B(x, \epsilon)}\left( - G(x, y)-\frac{V}{2\pi} \log \epsilon + V m(y)+u_\epsilon(y) \right)\frac{\dif V(y)}{V}
 \end{multline}
 Now, since we have normalized the Green's function to integrate to $0$, the contribution of the integral of $G(x,y)$ over $M\setminus B(x, \epsilon)$ will be $O(\epsilon)$, and we can lump this contribution together with that of the harmonic function $u_\epsilon(y)$, for which the $C^0$ norm is also $O(\epsilon)$:
 \be
 \mathbb{E}[\text{ duration Game I}]=O(\epsilon)-\frac{V}{2\pi} \log \epsilon +\int\limits_{M\setminus B(x, \epsilon)} m(y) \dif V(y)=-\frac{V}{2\pi} \log \epsilon +\int\limits_{M} m(y) \dif V(y) +O(\epsilon)
\ee
 Note that this term is clearly constant with respect to $x$, since anything involving $x$ has been integrated away (and we can estimate the $O(\epsilon)$ term uniformly).  
 
Now, in order to obtain the expected duration of Game II, we do almost the same calculation, however since the seeker is randomly chosen, we must integrate with respect to  $\frac{\dif V(x)}{V}$.  
\bea
\int\limits_M H^{\epsilon}(x,y) \frac{\dif V(x)}{V}&=&\int\limits_{M\setminus B(y, \epsilon)} H^{\epsilon}(x,y) \frac{\dif V(x)}{V}\\
&=&\hspace{-8pt}\int\limits_{M\setminus B(y, \epsilon)} \hspace{-8pt} \left(  - G(x, y)-\frac{V}{2\pi} \log \epsilon + V m(y)+u_\epsilon(y) \right)\frac{\dif V(x)}{V}\\
& = &-\frac{V}{2\pi} \log \epsilon + V m(y) + O(\epsilon)
\eea
\end{proof}
Now, recalling that the integral of the Robin's mass is related to the regularized trace on a surface, as stated in Proposition \ref{p.massisspecdensity}, so we see that we have developed a probabilistic interpretation for the regularized trace  and its density by way of the two hide-and-seek games.

\begin{cor}  On surfaces, the expected duration of Game I  is  the regularized trace for $\Delta^{-1}$,  up to an $\epsilon$-regularization.  The expected duration of Game II on surfaces is the density for the $\epsilon$-regularized trace.\end{cor}
\paragraph{Application}
Now, we can use the probabilistic intuition to consider how the regularized trace for different surfaces might compare.  At  a strictly intuitive level, we indicate how one might apply this to a family of rectangular tori with a flat metric, where area is fixed, and the lengths are getting longer and longer.  If one thinks about the regularized trace in terms of the hide-and-seek game, it seems quite plausible that as the rectangle is becoming longer (and therefore skinnier since the area is fixed), there will be more and more area that is farther away from a given point.  Therefore, it will take longer to reach the regions that are increasingly far away, so as the fundamental rectangles become longer and skinnier, the regularized trace becomes larger.  As demonstrated in \cite{skinnytori}, by connecting the regularized trace with the Dedekind eta function, one can verify that this behavior holds.  Moreover, the probabilistic intuition suggests that if one `blows a bubble' into a long-skinny torus, than most of the area will be closer to a given point, so the regularized trace should be smaller, and in particular  it should be close to that of the sphere.  This phenomenon does indeed occur, and it is described in \cite{skinnytori}.      

\bibliographystyle{hplain}

\bibliography{hide}
  
\end{document}